\documentclass[12pt, reqno]{amsart}
\usepackage{amsmath, amsthm, amscd, amsfonts, amsthm, amssymb, graphicx}
\usepackage[bookmarksnumbered, colorlinks, plainpages]{hyperref}

\input{mathrsfs.sty}

\newtheorem{theorem}{Theorem}[section]
\newtheorem{lemma}[theorem]{Lemma}
\newtheorem{proposition}[theorem]{Proposition}
\newtheorem{corollary}[theorem]{Corollary}
\theoremstyle{definition}
\newtheorem{definition}[theorem]{Definition}
\newtheorem{example}[theorem]{Example}

\numberwithin{equation}{section}

\begin{document}
	\title[Some inequalities on $h$-convex functions]{Some inequalities on $h$-convex functions}
	
	\author[ Abbasi, Morassaei, Mirzapour]{M. Abbasi$^1$, A. Morassaei$^1$ and F. Mirzapour$^1$}
	\vspace{-2cm}
	\address{$^{1}$ Department of Mathematics, Faculty of Sciences, University of Zanjan,  University Blvd., Zanjan 45371-38791, IRAN.}
	
	\email{mostafa.abasi5754@gmail.com}\email{morassaei@znu.ac.ir}\email{f.mirza@znu.ac.ir}
	
	
\begin{abstract}
In this paper, we state some characterizations of $h$-convex function is defined on a convex set in a linear space. By doing so, we extend the Jensen-Mercer inequality for $h$-convex function. We will also define $h$-convex function for operators on a Hilbert space and present the operator version of the Jensen-Mercer inequality. Lastly, we propound the complementary inequality of Jensen's inequality for $h$-convex functions.
\end{abstract} \maketitle

\section{Introduction}
Assume that $I$ is an interval in $\mathbb{R}$. Let us recall definitions of some special classes of functions.

We say that \cite{GL} $f:I\to\mathbb R$ is a Godunova-Levin function, or that $f$ belongs to the class $Q(I)$ if $f$ is non-negative and for all $x, y\in I$ and $t\in(0,1)$ we have
$$
f(t x+(1-t)y) \le \frac{f(x)}{t}+\frac{f(y)}{1-t}\,.
$$
For $s\in(0, 1]$, a function $f:[0, \infty)\to[0, \infty)$ is said to be $s$-convex function, or that $f$ belongs to the class $K_s^2$, if 
$$
f(t x+(1-t)y) \le t^sf(x)+(1-t)^sf(y)
$$
for every $x, y\in [0, \infty)$ and $t\in[0, 1]$ (see \cite{Bre}). Also, we say that $f:I\to[0, \infty)$ is a $P$-function \cite{DPP}, or that $f$ belongs to the class $P(I)$, if for all $x, y\in I$ and $t\in[0, 1]$ we have 
$$
f(t x+(1-t)y) \le f(x)+f(y)\,.
$$
Throughout this paper, suppose that $I$ and $J$ are intervals in $\mathbb{R}$, $(0, 1)\subseteq J$ and functions $h$ and $f$ are real non-negative functions defined on $J$ and $I$, respectively. 

In \cite{VAR}, Varo\v{s}anec defined  the $h$ -convex function as follows:\\
Let $h:J\subseteq\mathbb{R}\to\mathbb{R}$ be a non-negative function, $h\not\equiv 0$. We say that $f:I\to\mathbb{R}$ is a $h$-convex function, or that $f$ belongs to the class $SX(h, I)$, if $f$ is non-negative and for all $x, y\in I$, $t\in(0, 1)$ we have 
\begin{equation}\label{eq1111}
f(t x+(1-t)y) \le h(t)f(x)+h(1-t)f(y)\,.
\end{equation}
If inequality \eqref{eq1111} is reversed, then $f$ is said to be $h$-concave, that isو $f\in SV(h, I)$.

If $h(t)=t$, then all non-negative convex functions belong to $SX(h, I)$  and all non-negative concave functions belong to $SV(h, I)$. If $h(t)=\frac{1}{t}$, then $SX(h, I)=Q(I)$; if $h(t)=1$, then $SX(h, I)\supseteq P(I)$; and if $h(t)=t^s$, where $s\in(0, 1)$, then $SX(h, I)\supseteq K_s^2$.

A function $h:J\to\mathbb{R}$ is said to be a \textit{super-additive function} if 
\begin{equation}\label{SA}
h(x+y) \ge h(x)+h(y)\,,
\end{equation}
for all $x, y\in J$. If inequality \eqref{SA} is reversed, then $h$ is said to be a \textit{sub-additive function}. If the equality holds in \eqref{SA}, then $h$ is said to be an \textit{additive function}.

Function $h$ is called a \textit{super-multiplicative function} if 
\begin{equation}\label{SM}
h(xy) \ge h(x)h(y)\,,
\end{equation}
for all $x, y\in J$ \cite{VAR}. If inequality \eqref{SM} is reversed, then $h$ is called a \textit{sub-multiplicative function}. If the equality holds in \eqref{SM}, then $h$ is called a \textit{multiplicative function}.
\begin{example}\cite{VAR}
	Consider the function $h:[0, +\infty)\to\mathbb{R}$ by $h(x)=(c+x)^{p-1}$. If $c=0$, then the function $h$ is multiplicative. If $c\ge 1$, then for $p\in(0, 1)$ the function $h$ is super-multiplicative and for $p>1$ the function $h$ is sub-multiplicative.
\end{example}
\section{Preliminaries}
In what follows we assume that $\mathcal H$ and $\mathcal K$ are Hilbert spaces, $\mathbb B(\mathcal H)$ and $\mathbb B(\mathcal K)$ are $C^*$-algebras of all bounded linear operators on the appropriate Hilbert space with identities $I_{\mathcal H}$ and $I_{\mathcal K}$, $\mathbb{B}_h(\mathcal{H})$ denotes the algebra of all self-adjoint operators in $\mathbb{B}(\mathcal{H})$. An  operator $A\in\mathbb{B}_h(\mathcal{H})$ is called \textit{positive}, if $\langle Ax,x\rangle\ge0$ holds for every $x\in \mathcal{H}$ and then we write $A\ge0$. For $A,B\in\mathbb{B}_h(\mathcal{H})$, we say $A\le B$ if $B-A \ge 0$. We write $A>0$ and say $A$ is \textit{strictly positive operator}, if $A$ is a positive invertible operator. Let
$f$ be a continuous real valued function defined on an interval $I$. The function
$f$ is called \textit{operator monotone} if $A\le B$ implies $f(A)\le f(B)$ for all $A,B$ with
spectra in $I$. A function $f$ is said to be \textit{operator convex} on $I$ if
$$
f(t A + (1 - t)B) \le t f(A) + (1 - t)f(B)\,,
$$
for all $A,B\in\mathbb{B}_h(\mathcal{H})$ with spectra in $I$ and all $t\in[0,1]$.
A map $\Phi:\mathbb{B}(\mathcal{H})\to\mathbb{B}(\mathcal{K})$ is called \textit{positive} if $\Phi(A)\geq 0$, whenever $A\geq 0$ and is said to be \textit{normalized} if $\Phi(I_\mathcal{H})=I_\mathcal{K}$.
We denote by $\mathbf{P}[\mathbb{B}(\mathcal{H}),\mathbb{B}(\mathcal{K})]$ the set of all positive linear maps $\Phi:\mathbb{B}(\mathcal{H})\to\mathbb{B}(\mathcal{K})$ and by $\mathbf{P}_N[\mathbb{B}(\mathcal{H}),\mathbb{B}(\mathcal{K})]$ the set of all normalized positive linear maps $\Phi\in\mathbf{P}[\mathbb{B}(\mathcal{H}),\mathbb{B}(\mathcal{K})]$. If $\Phi \in \mathbf{P}_N[\mathbb{B}(\mathcal{H}),\mathbb{B}(\mathcal{K})]$ and $f$ is an operator convex function on an interval $I$, then
\begin{equation}\label{jen}
f(\Phi(A))\le\Phi(f(A))\,\quad\quad(\mbox{Davis-Choi-Jensen's inequality})
\end{equation}
for every self-adjoint operator $A$ on $\mathcal{H}$, whose spectrum is contained in $I$, see \cite{FMPS}.
\section{Characterizations}
Assume that $C$ is a convex subset of a linear space $X$  and $f$ is an arbitrary real-valued function on $C$. The non-negative function $f:C\rightarrow\mathbb{R}$ is called \textit{$h$-convex function} on $C$, if $f(tx+(1-t)y)\le h(t)f(x)+h(1-t)f(y)$ for every $x,y \in C$ and $t\in[0, 1]$.

Let $x$ and $y$ be two fixed elements in $C$. Define the map $f_{x,y}$ as follows:
$$
f_{x,y}:[0,1]\rightarrow\mathbb{R}\,,\quad f_{x,y}(t)=f(tx+(1-t)y)\,.
$$
The following theorem is a characterization of $h$-convex functions.
\begin{theorem}[First characterization]
	With the above assumptions, the following statements are equivalent:
	\begin{enumerate}
		\item[(i)] $f$ is a $h$-convex function on $C$.
		\item[(ii)] The mapping $f_{x,y}$ is a $h$-convex function on $[0,1]$, for any $x,y\in C$.
	\end{enumerate}
\end{theorem}
\begin{proof}
	First, assume that (i) holds. Let $\alpha, \beta\in[0,1]$ such that $\alpha+\beta=1$ and $t_1, t_2\in[0, 1]$. Hence 
	\begin{align*}
	f_{x,y}(\alpha t_1+\beta t_2)=& f\big((\alpha t_1+\beta t_2)x+(1-\alpha t_1-\beta t_2)y\big)\\
	=&f\Big(\alpha\big(t_1x+(1-t_1)y\big)+\beta\big(t_2x+(1-t_2)y\big)\Big)\\
	\le & h(\alpha)f\big(t_1x+(1-t_1)y\big)+h(\beta)f\big(t_2x+(1-t_2)y\big)\\
	=& h(\alpha)f_{x,y}(t_1)+h(\beta)f_{x,y}(t_2)\,,
	\end{align*}
	this means that $f_{x,y}$ is a $h$-convex function on $[0,1]$.
	
	Conversely, suppose that (ii) holds. For $t\in[0, 1]$ and $x, y\in C$, we have
	\begin{align*}
	f\big(tx+(1-t)y\big)=&f_{x,y}(t)=f_{x,y}\big((1-t)0+t1\big)\\
	\le & h(1-t)f_{x,y}(0)+h(t)f_{x,y}(1)\\
	=& h(1-t)f(y)+h(t)f(x)\,,
	\end{align*}
	that is, $f$ is a $h$-convex function on $C$.
\end{proof}
Now, for fixed $t\in[0, 1]$, we define the function $f_t:C^2\rightarrow\mathbb{R}$ by $f_t(x, y)=f(tx+(1-t)y)$. 

In the next theorem, we state a new characterization of $h$-convex functions.
\begin{theorem}[Second characterization]
	The following statements of $h$-convex functions hold:
	\begin{enumerate}
		\item[(i)] If $f$ is a $h$-convex function on $C$, then $f_t$ is a $h$-convex function on $C^2$ for every $t\in[0, 1]$.
		\item[(ii)] If $C$ is a cone in $X$ and $f_t$ is a $h$-convex function on $C^2$ for every $t\in(0, 1)$, then $f$ is a $h$-convex function on $C$.
	\end{enumerate}
\end{theorem}
\begin{proof}
	(i) For fixed $t\in [0, 1]$ and $(x, y), (u, v)\in C^2$. Then for every $\alpha\in[0, 1]$ 
	\begin{align*}
	f_t\big(\alpha(x,y)+(1-\alpha)(u, v)\big)=&f_t\big(\alpha x+(1-\alpha)u,\alpha y+(1-\alpha)v\big)\\
	=& f\big(t(\alpha x+(1-\alpha)u)+(1-t)(\alpha y+(1-\alpha)v)\big)\\
	=& f\big(\alpha(tx+(1-t)y)+(1-\alpha)(tu+(1-t)v)\big)\\
	\le& h(\alpha)f\big(tx+(1-t)y\big)+h(1-\alpha)f\big(tu+(1-t)v\big)\\
	=&h(\alpha)f_t(x,y)+h(1-\alpha)f_t(u,v)\,,
	\end{align*}
	that is, $f_t$ is a $h$-convex function on $C^2$.
	
	(ii) Let $x,y \in C$ and $t\in(0, 1)$. Since $C$ is cone in $X$, $C+C\subseteq C$ and $\alpha C\subseteq C$ for every $\alpha\ge 0$, then $t^{-1}x, (1-t)^{-1}y\in C$ and $(t^{-1}x, 0),(0,(1-t)^{-1}y)\in C^2$. On the other hand, by $h$-convexity of $f_t$ on $C^2$, we have
	\begin{align*}
	f\big(tx+(1-t)y\big)=&f_t(x, y)\\
	=&f_t\big(t(t^{-1}x, 0)+(1-t)(0,(1-t)^{-1}y)\big)\\
	\le& h(t)f_t(t^{-1}x, 0)+h(1-t)f_t(0,(1-t)^{-1}y)\\
	=& h(t)f(x)+h(1-t)f(y)\,,
	\end{align*}
	therefore, $f$ is a $h$-convex function on $C$.
	
\end{proof}
\begin{theorem}[Third characterization]
	Let $h$ be a strictly positive multiplicative function, then the following statements are equivalent:
	\begin{enumerate}
		\item[(i)] $f$ is a $h$-convex function.
		\item[(ii)] If $(1+s)x-sy\in C$, for every $x, y\in C$ and $s \ge 0$, then 
		\begin{equation}
		f\big((1+s)x-sy\big) \ge h(1+s)f(x)-h(s)f(y)\,.
		\end{equation}
	\end{enumerate}
\end{theorem}
\begin{proof}
	First, note that multiplicity of $h$ implies that $h\left(\frac{1}{t}\right)=\frac{1}{h(t)}$, for every $t > 0$ and $h(1)=1$.
	
	Assume that (i) holds. By using $x=\frac{1}{s+1}[(1+s)x-sy]+\frac{s}{s+1}y$, we have 
	\begin{align*}
	f(x)=& f\left(\frac{1}{s+1}[(1+s)x-sy]+\frac{s}{s+1}y\right)\\
	\le &h\left(\frac{1}{s+1}\right)f\big[(1+s)x-sy\big]+h\left(\frac{s}{s+1}\right)f(y)\\
	=&\frac{h(1)}{h(s+1)}f\big[(1+s)x-sy\big]+\frac{h(s)}{h(s+1)}f(y)\qquad(\text{by multiplicity of}~ h)
	\end{align*}
	and therefore,
	$$
	f\big[(1+s)x-sy\big] \ge h(1+s)f(x)-h(s)f(y)\,.
	$$
	Now, suppose that (ii) holds. If $\alpha, \beta\in[0, 1]$ and $\alpha+\beta=1$, then there exists $s \ge 0$ such that $\alpha=\frac{1}{s+1}$ and $\beta=\frac{s}{s+1}$. Put $z=\alpha x+\beta y$, hence $x=(1+s)z-sy$ and so 
	$$
	f(x)=f\big((1+s)z-sy\big) \ge h(1+s)f(z)-h(s)f(y)\,.
	$$
	Consequently, 
	\begin{align*}
	f(\alpha x+\beta y)= & f(z)\le \frac{1}{h(1+s)}f(x)+\frac{h(s)}{h(1+s)}f(y)\\
	= & h\left(\frac{1}{1+s}\right)f(x)+h\left(\frac{s}{1+s}\right)f(y)\\
	= & h(\alpha)f(x)+h(\beta)f(y)\,,
	\end{align*}
	this show that $f$ is $h$-convex function.
\end{proof}
\begin{theorem}\label{t3.4}
	\begin{enumerate}
		\item[(i)] Assume that $X$ is a real vector space and $f:X\rightarrow\mathbb{R}$ is an even $h$-convex function. Then 
		\begin{align}\label{e333}
		\frac{f((1-2t)x)+f((2t-1)y)}{h(t)+h(1-t)}\le & f\big((1-t)x+ty\big)+f\big(tx+(1-t)y\big)\nonumber\\
		\le & [h(t)+h(1-t)][f(x)+f(y)]\,.
		\end{align}
		\item[(ii)] Let $X$ be a topological vector space, $h$ be an integrable strictly positive function and $f$ be a continuous even $h$-convex function, then
		\begin{equation}\label{e334}
		\frac{1}{2}\int_{0}^{1}[f(tx)+f(ty)]~\mathrm{d}t \le \int_0^1[h(t)+h(1-t)]f\big(tx+(1-t)y\big)~\mathrm{d}t\,.
		\end{equation}
		In addition, if $h$ is super-additive, then
		\begin{equation}\label{e335}
		\frac{1}{2h(1)\left(\int_0^1h(t)~\mathrm{d}t\right)}\int_0^1[f(tx)+f(ty)]~\mathrm{d}t \le f(x)+f(y)\,.
		\end{equation}
	\end{enumerate}
\end{theorem}
\begin{proof}
	(i) Since $f$ is a $h$-convex function, hence the right inequality of \eqref{e333} is clear. Set $a=(1-t)x+ty$ and $b=-tx-(1-t)y$, we have $(1-t)a+tb=(1-2t)x$ and $ta+(1-t)b=(2t-1)y$. Therefore
	\begin{align*}
	f&((1-2t)x)+f((2t-1)y)\\
	=&f\big((1-t)a+tb\big)+f\big(ta+(1-t)b\big)\\
	\le & [h(t)+h(1-t)][f(a)+f(b)]\hspace{3.5cm}{\footnotesize(\text{by the right inequality of \eqref{e333}})}\\
	=& [h(t)+h(1-t)]\big[f((1-t)x+ty)+f(-tx-(1-t)y)\big]\\
	=& [h(t)+h(1-t)]\big[f((1-t)x+ty)+f(tx+(1-t)y)\big]\,.\qquad{\footnotesize(f~\text{is a even function})}
	\end{align*}
	We therefore deduce the desired inequality in \eqref{e333}.\\
	(ii) By using \eqref{e333} inequality, we get
	{\footnotesize
		\begin{equation}\label{e3333}
		f((1-2t)x)+f((2t-1)y)\le [h(t)+h(1-t)]\left[f\big((1-t)x+ty\big)+f\big(tx+(1-t)y\big)\right]\,.
		\end{equation}
	}
	Integrating each side of \eqref{e3333}, we have
	{\footnotesize
		\begin{align}\label{e390}
		\int_0^1&f((1-2t)x)~\mathrm{d}t+\int_0^1f((2t-1)y)~\mathrm{d}t\nonumber\\
		&\le \int_0^1[h(t)+h(1-t)]f\big((1-t)x+ty\big)~\mathrm{d}t+\int_0^1[h(t)+h(1-t)]f\big(tx+(1-t)y\big)~\mathrm{d}t\nonumber\\
		&=2\int_0^1[h(t)+h(1-t)]f\big(tx+(1-t)y\big)~\mathrm{d}t\,.
		\end{align}
	}
	Since $f$ is even and by changing of variables $u=1-2t$, yield
	$$
	\int_0^1f((1-2t)x)~\mathrm{d}t=\int_0^1f(tx)~\mathrm{d}t\,,
	$$
	and similarly
	$$
	\int_0^1f((2t-1)y)~\mathrm{d}t=\int_0^1f(ty)~\mathrm{d}t\,.
	$$
	Consequently, by \eqref{e390}, the proof of \eqref{e334} completes.
	
	Now, assume that $h$ is supper-additive. Hence by \eqref{e334}, we have
	\begin{align*}
	\frac{1}{2}&\int_{0}^{1}[f(tx)+f(ty)]~\mathrm{d}t\\
	\le& \int_0^1[h(t)+h(1-t)]f\big(tx+(1-t)y\big)~\mathrm{d}t\\
	\le & h(1)\int_0^1\big[h(t)f(x)+h(1-t)f(y)\big]~\mathrm{d}t\quad\footnotesize{(h ~\text{is super-additive and } f ~\text{is $h$-convex})}\\
	=& h(1)\left[f(x)\left(\int_0^1h(t)~\mathrm{d}t\right)+f(y)\left(\int_0^1h(1-t)~\mathrm{d}t\right)\right]\\
	=& h(1)\left(\int_0^1h(t)~\mathrm{d}t\right)[f(x)+f(y)]\,,
	\end{align*}
	this show that \eqref{e335} holds.
\end{proof}
\begin{corollary}
	\begin{enumerate}
		\item[(i)] Assume that $X$ is a real vector space and $f:X\rightarrow\mathbb{R}$ is an even convex function. Then 
		\begin{align}
		f((1-2t)x)+f((2t-1)y)\le & f\big((1-t)x+ty\big)+f\big(tx+(1-t)y\big)\nonumber\\
		\le & f(x)+f(y)\,.
		\end{align}
		\item[(ii)] Let $X$ be a topological vector space and $f$ be a continuous even convex function, then
		\begin{equation}
		\frac{1}{2}\int_{0}^{1}[f(tx)+f(ty)]~\mathrm{d}t \le \int_0^1f\big(tx+(1-t)y\big)~\mathrm{d}t\le f(x)+f(y)\,.
		\end{equation}
	\end{enumerate}
\end{corollary}
\begin{proof}
	Enough put in Theorem \ref{t3.4}, $h(t)=t$.
\end{proof}
\begin{corollary}\cite[Lemma 3.2]{RHM}
	\begin{enumerate}
		\item[(i)] Assume that $X$ is a real vector space and $f:X\rightarrow\mathbb{R}$ is an even function in $P(I)$. Then 
		\begin{align}
		\frac{f((1-2t)x)+f((2t-1)y)}{2}\le & f\big((1-t)x+ty\big)+f\big(tx+(1-t)y\big)\nonumber\\
		\le & 2\big(f(x)+f(y)\big)\,.
		\end{align}
		\item[(ii)] Let $X$ be a topological vector space and $f$ be a continuous even function in $P(I)$, then
		\begin{equation}
		\frac{1}{4}\int_{0}^{1}[f(tx)+f(ty)]~\mathrm{d}t \le \int_0^1f\big(tx+(1-t)y\big)~\mathrm{d}t\le f(x)+f(y)\,.
		\end{equation}
	\end{enumerate}
\end{corollary}
\begin{proof}
	In Theorem \ref{t3.4}, put $h(t)=1$.
\end{proof}
\begin{example}\cite[Theorem 3.3]{RHM}
	Let $(X, \|\cdot\|)$ be a normed space, $x, y \in X$ and $0 < p <1$. Since $f(x)=\|x\|^p$ is an even continuous $P$-convex function, we have the following Hermit-Hadamard inequality 
	$$
	\frac{\|x\|^p+\|y\|^p}{4(p+1)}\le \int_0^1\|(1-t)x+ty\|^p\mathrm{d}t \le \|x\|^p+\|y\|^p\,.
	$$
\end{example}
\section{Jensen-Mercer type inequality}
In \cite{AMDM}, Mercer proved that
\begin{equation}\label{e1}
f\left(x_1+x_n-\sum_{j=1}^nt_jx_j\right) \le f(x_1)+f(x_n)-\sum_{j=1}^n t_jf(x_j)\,.
\end{equation}
where $x_j$'s also satisfy in the condition $0<x_1 \le x_2 \le \cdots \le x_n$, $t_j \ge 0$ and $\sum_{j=1}^nt_j=1$.

In this section, we present the Jensen-Mercer inequality for $h$-convex functions.
\begin{theorem}\cite[Theorem 19]{VAR}
	Let $t_1, \cdots, t_n$ be positive real numbers $(n\ge 2)$. If $h$ is a non-negative super-multiplicative function, $f$ is a $h$-convex function on $I$ and $x_1, \cdots, x_n\in I$, then
	\begin{equation}
	f\left(\frac{1}{T_n}\sum_{j=1}^nt_jx_j\right)\le\sum_{j=1}^nh\left(\frac{t_j}{T_n}\right)f(x_j)\,,
	\end{equation}
	where $T_n=\sum_{j=1}^nt_j$.
\end{theorem}
\begin{lemma}\label{l4.2}
	Let $0<x \le y$ and $f$ be a $h$-convex function, then for every $z\in[x, y]$, there exists $\lambda\in[0, 1]$ such that
	\begin{equation}\label{e44}
	f(x+y-z) \le [h(\lambda)+h(1-\lambda)][f(x)+f(y)]-f(z)\,.
	\end{equation}
	Moreover, if $h$ is super-additive, then 
	$$
	f(x+y-z) \le h(1)[f(x)+f(y)]-f(z)\,.
	$$
\end{lemma}
\begin{proof}
	Since $z\in[x, y]$, there exists $\lambda\in[0,1]$ such that
	$$
	z= \lambda x+(1-\lambda)y\,.
	$$
	By using $h$-convexity of $f$, we have
	\begin{align*}
	f(x+y-z)&= f\big((1-\lambda)x+\lambda y\big)\\
	&\le h(1-\lambda)f(x)+h(\lambda)f(y)\\
	&= [h(\lambda)+h(1-\lambda)][f(x)+f(y)]-[h(\lambda)f(x)+h(1-\lambda)f(y)]\\
	&\le  [h(\lambda)+h(1-\lambda)][f(x)+f(y)]-f\big(\lambda x+(1-\lambda)y\big)\\
	&=  [h(\lambda)+h(1-\lambda)][f(x)+f(y)]-f(z)\,.
	\end{align*}
	If $h$ is super-additive, then $h(\lambda)+h(1-\lambda)\le h(1)$. So the end part of theorem holds.
\end{proof}
\begin{theorem}
	Let $f$ be a $h$-convex function on an interval containing the $x_j~(j=1, \cdots, n)$ such that $0<x_1\le \cdots \le x_n$, then
	{\footnotesize
		$$
		f\left(x_1+x_n-\sum_{j=1}^{n}t_jx_j\right)\le \left(\sum_{j=1}^nh(t_j)[h(\lambda_j)+h(1-\lambda_j)]\right)\big(f(x_1)+f(x_n)\big)-\sum_{j=1}^nh(t_j)f(x_j)\,,
		$$
	}
	where for every $j=1,\cdots,n$, there exists $\lambda_j\in[0,1]$ such that  $x_j=\lambda_jx_1+(1-\lambda_j)x_n$.
\end{theorem}
\begin{proof}
	With the above assumption, we have
	\begin{align*}
	f & \left(x_1+x_n-\sum_{j=1}^{n}t_jx_j\right)\\
	&=f\left(\sum_{j=1}^nt_j(x_1+x_n-x_j)\right)\\
	&\le \sum_{j=1}^nh(t_j)f(x_1+x_n-x_j)\\
	&\le \sum_{j=1}^nh(t_j)\Big([h(\lambda_j)+h(1-\lambda_j)]\big(f(x_1)+f(x_n)\big)-f(x_j)\Big)~{\footnotesize(\text{by Lemma \ref{l4.2}})}\\
	&=\sum_{j=1}^nh(t_j)[h(\lambda_j)+h(1-\lambda_j)]\big(f(x_1)+f(x_n)\big)-\sum_{j=1}^nh(t_j)f(x_j)\,,
	\end{align*}
	this completes the proof.
\end{proof}
\begin{corollary}
	With the assumptions of previous theorem, if $h$ is a super-additive function such that for every probability vector $(t_1,\cdots,t_n)$, $\sum_{j=1}^nh(t_j)\le 1$, then
	$$
	f\left(x_1+x_n-\sum_{j=1}^{n}t_jx_j\right)\le h(1)\big(f(x_1)+f(x_n)\big)-\sum_{j=1}^nh(t_j)f(x_j)\,.
	$$
	Moreover, if $h$ is multiplicative, then
	$$
	f\left(x_1+x_n-\sum_{j=1}^{n}t_jx_j\right)\le f(x_1)+f(x_n)-\sum_{j=1}^nh(t_j)f(x_j)\,.
	$$
\end{corollary}
\section{Operator $h$-convex functions}
In this section, we present the definition of operator $h$-convex function for operators acting on a Hilbert space. 
\begin{definition}
	Let $h:J\subseteq\mathbb{R}\to\mathbb{R}$ be a non-negative function, $h\not\equiv 0$. We say that $f:I\to\mathbb{R}$ is an \textit{operator $h$-convex function}, if $f$ is non-negative and for all $A, B\in\mathbb{B}(\mathcal{H})$ with $\sigma(A), \sigma(B)\subseteq I$ and $t\in(0, 1)$, 
	\begin{equation}\label{eq111}
	f\left(t A+(1-t)B\right) \le h(t)f(A)+h(1-t)f(B)\,,
	\end{equation}
	where $\sigma(A)$ and $\sigma(B)$ are spectrum of $A$ and $B$, respectively.
	
	If inequality \eqref{eq111} is reversed, then $f$ is said to be \textit{operator $h$-concave}.
\end{definition}
If $t=\frac{1}{2}$ in \eqref{eq111}, then $f$ is called $h$-\textit{mid-convex function}.
\begin{example}
	Assume that $h$ is a function on $[0, \infty)$ such that $h(t) \ge t$ and $f(t)=t^2$ on an interval $I\subseteq\mathbb{R}$. Then $f$ is operator $h$-mid-convex function. Because, 
	\begin{align*}
	&h\left(\frac{1}{2}\right)\left(A^2+B^2\right)-\left(\frac{A+B}{2}\right)^2\\
	&=h\left(\frac{1}{2}\right)\left(A^2+B^2\right)-\frac{A^2+AB+BA+B^2}{4}\\
	&=\frac{\left(4h(1/2)-1\right)A^2-AB-BA+\left(4h(1/2)-1\right)B^2}{4}\\
	&\ge\frac{1}{4}\left(A^2-AB-BA-B^2\right)\\
	&=\frac{1}{4}(A-B)^2 \ge 0\,.
	\end{align*}
\end{example}
Now, we can prove the following theorem as Theorem 1.9 in \cite{FMPS}. So, we omit the proof it.
\begin{theorem}[Jensen's type operator inequality]\label{jtoi}
	Let $\mathcal{H}$ and $\mathcal{K}$ be Hilbert space. Assume that $h$ is non-negative super-mutiplicative function and $f$ is a real valued function on an interval $I\subseteq\mathbb{R}$. Suppose that $A, A_j\in\mathbb{B}_h(\mathcal{H})$ such that $\sigma(A), \sigma(A_j)\subseteq I~ (j=1,\cdots,n)$. Then the following conditions are equivalent:
	\begin{enumerate}
		\item[(i)] $f$ is operator $h$-mid-convex on $I$;
		\item[(ii)] $f(C^*AC)\le 2h(\frac{1}{2})C^*f(A)C$ for every self-adjoint operator $A:\mathcal{H}\to \mathcal{H}$ and isometry $C:\mathcal{K}\to\mathcal{H}$, i.e.; $C^*C=1_{\mathcal{K}}$;
		\item[(iii)] $f(C^*AC)\le 2h(\frac{1}{2})C^*f(A)C$ for every self-adjoint operator $A:\mathcal{H}\to \mathcal{H}$ and isometry $C:\mathcal{H}\to \mathcal{H}$;
		\item[(iv)] $f\left(\sum_{j=1}^nC_j^*A_jC_j\right)\le2h(\frac{1}{2})\sum_{j=1}^nC_j^*f(A_j)C_j$ for every self-adjoint operator $A_j:\mathcal{H}\to \mathcal{H}$ and bounded linear operators $C_j:\mathcal{K}\to \mathcal{H}$; with $\sum_{j=1}^nC_j^*C_j=1_{\mathcal{K}}\quad(j=1,\cdots,n)$;
		\item[(v)] $f\left(\sum_{j=1}^nC_j^*A_jC_j\right)\le2h(\frac{1}{2})\sum_{j=1}^nC_j^*f(A_j)C_j$ for every self-adjoint operator $A_j:\mathcal{H}\to \mathcal{H}$ and bounded linear operators $C_j:\mathcal{H}\to \mathcal{H}$; with $\sum_{j=1}^nC_j^*C_j=1_{\mathcal{H}}\quad(j=1,\cdots,n)$;
		\item[(vi)] $f\left(\sum_{j=1}^nP_jA_jP_j\right)\le2h(\frac{1}{2})\sum_{j=1}^nP_jf(A_j)P_j$ for every self-adjoint operator $A_j:\mathcal{H}\to \mathcal{H}$ and projection $P_j:\mathcal{H}\to \mathcal{H}$; with $\sum_{j=1}^nP_j=1_{\mathcal{H}}\quad(j=1,\cdots,n)$.
	\end{enumerate}
\end{theorem}
Using an idea of \cite{FMPS} we prove the following result.
\begin{theorem}[Davis-Choi-Jensen's inequality]
	Let $\Phi$ be a normalized positive linear map and $f$ be an operator $h$-convex function on an interval $I$, then
	\begin{equation}\label{eqDCJ}
	f(\Phi(A)) \le 2h\left(\frac{1}{2}\right) \Phi(f(A))\,,
	\end{equation}
	for every self-adjoint  operator $A$ with $\sigma(A)\subseteq I$.
\end{theorem}
\begin{proof}
	We know that a self-adjoint operator $A$ can be approximated uniformly by a simple function $A'=\sum_{j}t_jE_j$ where $\{E_j\}$ is a decomposition of the unit $I_{\mathcal H}$. By using normality of $\Phi$, we get $\sum_{j}\Phi(E_j)=I_{\mathcal K}$. By applying (iv) of Theorem \ref{jtoi} to $C_j=\sqrt{\Phi(E_j)}$, we have 
	\begin{align*}
	f(\Phi(A'))= & f\left(\sum_{j}t_j\Phi(E_j)\right)=f\left(\sum_{j}C_jt_jC_j\right)\\
	\le & 2h\left(\frac{1}{2}\right)\sum_{j}C_jf(t_j)C_j=2h\left(\frac{1}{2}\right)\sum_{j}f(t_j)\Phi(E_j)\\
	= & 2h\left(\frac{1}{2}\right)\Phi\left(\sum_{j}f(t_j)E_j\right)=2h\left(\frac{1}{2}\right)\Phi(f(A'))\,.
	\end{align*}
	Since $\Phi$ is continuous, the proof is complete.
\end{proof}
\begin{theorem}\label{t5.3}
	Let $t_1, t_2, \cdots, t_n$ be positive real numbers $(n \ge 2)$ such that $\sum_{j=1}^nt_j=1$. If $h$ is non-negative super-multiplicative function and if $f$ is $h$-convex function on an interval $I\subseteq\mathbb{R}$, $A_1, \cdots, A_n$ are self-adjoint operators in $\mathbb{B}(\mathcal{H})$ such that $\sigma(A_j)\subseteq I$, then 
	\begin{equation}\label{Jen}
	f\left(\sum_{j=1}^nt_jA_j\right) \le \sum_{j=1}^{n}h(t_j)f(A_j)\,.
	\end{equation}
	If $h$ is sub-multiplicative and $f$ is operator $h$-concave on $I$, then inequality \eqref{Jen} is reversed. 
\end{theorem}
\begin{proof}
	We prove this theorem by induction on $n$. If $n=2$, then inequality \eqref{Jen} is equivalent to inequality \eqref{eq111} with $t=t_1$ and $1-t=t_2$. Assume that inequality \eqref{Jen} holds for $n-1$. Then for $n$, we have 
	{\footnotesize
		\begin{align*}
		f\left(\sum_{j=1}^nt_jA_j\right)=&f\left(t_nA_n+\sum_{j=1}^{n-1}t_jA_j\right)\\
		=&f\left(t_nA_n+(t_1+\cdots+t_{n-1})\sum_{j=1}^{n-1}\frac{t_j}{t_1+\cdots+t_{n-1}}A_j\right)\\
		\le & h(t_n)f(A_n)+h(t_1+\cdots+t_{n-1})f\left(\sum_{j=1}^{n-1}\frac{t_j}{t_1+\cdots+t_{n-1}}A_j\right)\\
		\le & h(t_n)f(A_n)+h(t_1+\cdots+t_{n-1})\sum_{j=1}^{n-1}h\left(\frac{t_j}{t_1+\cdots+t_{n-1}}\right)f(A_j)\\
		\le & h(t_n)f(A_n)+\sum_{j=1}^{n-1}h(t_j)f(A_j)\\
		=&\sum_{j=1}^{n}h(t_j)f(A_j)\,.
		\end{align*}
	}
	This completes the proof.
\end{proof}
\begin{corollary}
	By assumptions of Theorem \ref{t5.3}, if $\Phi\in\mathbf{P}_N[\mathbb{B}(\mathcal{H}), \mathbb{B}(\mathcal{K})]$, then 
	\begin{equation}\label{Jen2}
	f\left(\sum_{j=1}^nt_j\Phi(A_j)\right) \le \sum_{j=1}^{n}2h\left(\frac{t_j}{2}\right)\Phi(f(A_j))\,.
	\end{equation}
	If $h$ is sub-multiplicative and $f$ is operator $h$-concave on $I$, then inequality \eqref{Jen2} is reversed.
\end{corollary}
\begin{proof}
	If $\Phi\in\mathbf{P}_N[\mathbb{B}(\mathcal{H}), \mathbb{B}(\mathcal{K})]$, then
	\begin{align*}
	f\left(\sum_{j=1}^nt_j\Phi(A_j)\right)\le &\sum_{j=1}^{n}h(t_j)f(\Phi(A_j))\\
	\le & \sum_{j=1}^n2h(t_j)h\left(\frac{1}{2}\right)\Phi(f(A_j))\quad(\text{by} ~\eqref{eqDCJ})\\
	\le & \sum_{j=1}^n2h\left(\frac{t_j}{2}\right)\Phi(f(A_j))\quad(\text{by super-multiplicity of } h)\,.
	\end{align*}
\end{proof}
For convenience, let $\varphi(t)$ be a real valued continuous function on the interval $[m, M]$. Define 
$$
\mu_\varphi=\frac{\varphi(M)-\varphi(m)}{M-m}\,,\quad \nu_\varphi=\frac{M\varphi(m)-m\varphi(M)}{M-m}\,.
$$
We remark that a straight line $\ell(t)=\mu_\varphi t+\nu_\varphi$ is a line thought two points $(m, \varphi(m))$ and $(M, \varphi(M))$.

Notice that, if $\varphi(t)=t$, then $\mu_\varphi=1$ and $\nu_\varphi=0$, if $\varphi(t)=1$, then $\mu_\varphi=0$ and $\nu_\varphi=1$, and if $\varphi(t)=\frac{1}{t}$, then $\mu_\varphi=-\frac{1}{mM}$ and $\nu_\varphi=\frac{m+M}{mM}$.

\begin{theorem}\label{t5.5}
	Let $A_1, A_2, \cdots, A_n\in\mathbb{B}(\mathcal H)$ be self-adjoint operators with spectra in $[m, M]$ for some scalars $m, M$ and $\Phi_1, \Phi_2, \cdots, \Phi_n\in \mathbf{P}_N[\mathbb{B}(\mathcal{H}), \mathbb{B}(\mathcal{K})]$ and $t_1, \cdots, t_n$ non-negative real numbers with $\sum_{j=1}^nt_j=1$. If $f$ on $[m, M]$ is operator $h$-convex function and $h$ on the interval $J$ is super-multiplicative operator convex function, then
	\begin{align}\label{eq5.5}
	f & \left(mI_{\mathcal K}+MI_{\mathcal K}-\sum_{j=1}^nt_j\Phi_j(A_j)\right)\\
	& \le h\left(\frac{\sum_{j=1}^nt_j\Phi_j(A_j)-mI_{\mathcal K}}{M-m}\right)f(m)+h\left(\frac{MI_{\mathcal K}-\sum_{j=1}^nt_j\Phi_j(A_j)}{M-m}\right)f(M)\nonumber\\
	& \le (\mu_h+2\nu_h)\big(f(m)+f(M)\big)I_{\mathcal K}-\sum_{j=1}^nh(t_j)\Phi_j(f(A_j))\,.\nonumber
	\end{align}
\end{theorem}
\begin{proof}
	Define the function $g:[m, M]\to\mathbb{R}$ by $g(t)=f(m+M-t)$. Since $f$ is continuous and $h$-convex on $[m, M]$, so the same is true for $g$. Consequently, for every $t\in[m, M]$ 
	\begin{equation}\label{e5.5}
	f(t) \le h\left(\frac{t-m}{M-m}\right)f(M)+h\left(\frac{M-t}{M-m}\right)f(m)\,,
	\end{equation}
	and 
	\begin{equation}\label{e5.6}
	g(t) \le h\left(\frac{t-m}{M-m}\right)g(M)+h\left(\frac{M-t}{M-m}\right)g(m)\,.
	\end{equation}
	Since $\sum_{j=1}^nt_j=1$, $\Phi_j(I_{\mathcal H})=I_{\mathcal K}$ and $mI_{\mathcal H} \le A_j \le MI_{\mathcal H}~(j=1, \cdots, n)$, we conclude that $mI_{\mathcal K} \le \sum_{j=1}^nt_j\Phi_j(A_j) \le MI_{\mathcal K}$. Now, by using functional calculus and \eqref{e5.5} and super-multiplicity of $h$, we get
	\begin{align}\label{eq5.7}
	h(t_j)f(A_j) \le & h(t_j)h\left(\frac{A_j-mI_{\mathcal H}}{M-m}\right)f(M)+h(t_j)h\left(\frac{MI_{\mathcal H}-A_j}{M-m}\right)f(m)\nonumber\\
	\le & h\left(\frac{t_jA_j-mt_jI_{\mathcal H}}{M-m}\right)f(M)+h\left(\frac{Mt_jI_{\mathcal H}-t_jA_j}{M-m}\right)f(m)\nonumber\\
	\le & \left[\mu_h\left(\frac{t_jA_j-mt_jI_{\mathcal H}}{M-m}\right)+\nu_hI_{\mathcal H}\right]f(M)\\
	&+\left[\mu_h\left(\frac{Mt_jI_{\mathcal H}-t_jA_j}{M-m}\right)+\nu_hI_{\mathcal H}\right]f(m)\nonumber\,.
	\end{align}
	Hence, by linearity of $\Phi_j$ for every $j=1, \cdots, n$ and the inequality \eqref{eq5.7}, we have 
	\begin{align*}
	h(t_j)\Phi_j(f(A_j)) \le & \left[\mu_h\left(\frac{t_j\Phi_j(A_j)-mt_jI_{\mathcal K}}{M-m}\right)+\nu_hI_{\mathcal K}\right]f(M)\\
	&+\left[\mu_h\left(\frac{Mt_jI_{\mathcal K}-t_j\Phi_j(A_j)}{M-m}\right)+\nu_hI_{\mathcal K}\right]f(m)\,.
	\end{align*}
	By summing of all $j=1, \cdots, n$
	\begin{align}\label{eq5.9}
	\sum_{j=1}^{n}h(t_j)\Phi_j(f(A_j)) \le & \left[\mu_h\left(\frac{\sum_{j=1}^{n}t_j\Phi_j(A_j)-mI_{\mathcal K}}{M-m}\right)+\nu_hI_{\mathcal K}\right]f(M)\\
	&+\left[\mu_h\left(\frac{MI_{\mathcal K}-\sum_{j=1}^{n}t_j\Phi_j(A_j)}{M-m}\right)+\nu_hI_{\mathcal K}\right]f(m)\nonumber\\
	= & \mu_h\left[\mu_f\sum_{j=1}^nt_j\Phi_j(A_j)+\nu_fI_{\mathcal K}\right]+\nu_h\big(f(m)+f(M)\big)I_{\mathcal K}\nonumber
	\end{align}
	Also, using similar way and the \eqref{e5.6} inequality, we have
	\begin{align*}
	g & \left(\sum_{j=1}^nt_j\Phi_j(A_j)\right)\\
	&\le h\left(\frac{\sum_{j=1}^nt_j\Phi_j(A_j)-mI_{\mathcal K}}{M-m}\right)f(m)+h\left(\frac{MI_{\mathcal K}-\sum_{j=1}^nt_j\Phi_j(A_j)}{M-m}\right)f(M)
	\end{align*}
	or equivalently,
	\begin{align*}
	f & \left(mI_{\mathcal K}+MI_{\mathcal K}-\sum_{j=1}^nt_j\Phi_j(A_j)\right)= g\left(\sum_{j=1}^nt_j\Phi_j(A_j)\right)\\
	\le &  h\left(\frac{\sum_{j=1}^nt_j\Phi_j(A_j)-mI_{\mathcal K}}{M-m}\right)f(m)+h\left(\frac{MI_{\mathcal K}-\sum_{j=1}^nt_j\Phi_j(A_j)}{M-m}\right)f(M)\\
	\le& \left[\mu_h\left(\frac{{\sum_{j=1}^nt_j\Phi_j(A_j)-mI_{\mathcal K}}}{M-m}\right)+\nu_hI_{\mathcal K}\right]f(m)\\
	&+\left[\mu_h\left(\frac{{MI_{\mathcal K}-\sum_{j=1}^nt_j\Phi_j(A_j)}}{M-m}\right)+\nu_hI_{\mathcal K}\right]f(M)\\
	= & \mu_h\Bigg[f(m)I_{\mathcal K}+f(M)I_{\mathcal K}-\Bigg(\frac{{MI_{\mathcal K}-\sum_{j=1}^nt_j\Phi_j(A_j)}}{M-m}f(m)\nonumber\\
	& +\frac{{\sum_{j=1}^nt_j\Phi_j(A_j)-mI_{\mathcal K}}}{M-m}f(M)\Bigg)\Bigg]+\nu_h\big(f(m)+f(M)\big)I_{\mathcal K}\nonumber\\
	= & \mu_h\left[\big(f(m)+f(M)\big)I_{\mathcal K}-\left(\mu_f\sum_{j=1}^nt_j\Phi_j(A_j)+\nu_fI_{\mathcal K}\right)\right]+\nu_h\big(f(m)+f(M)\big)I_{\mathcal K}\\
	\le & (\mu_h+2\nu_h)\big(f(m)+f(M)\big)I_{\mathcal K}-\sum_{j=1}^nh(t_j)\Phi_j(f(A_j))\,.\nonumber
	\end{align*}
	In the final inequality, we use the inequality \eqref{eq5.9}, and we obtain desired inequalities \eqref{eq5.5}.
\end{proof}
\begin{corollary}
	\begin{enumerate}
		\item[(i)] Let $h(t)=t$ in Theorem \ref{t5.5}, then we have
		\begin{align*}
		f & \left(mI_{\mathcal K}+MI_{\mathcal K}-\sum_{j=1}^nt_j\Phi_j(A_j)\right)\\
		& \le \left(\frac{\sum_{j=1}^nt_j\Phi_j(A_j)-mI_{\mathcal K}}{M-m}\right)f(m)+\left(\frac{MI_{\mathcal K}-\sum_{j=1}^nt_j\Phi_j(A_j)}{M-m}\right)f(M)\\
		& \le f(m)I_{\mathcal K}+f(M)I_{\mathcal K}-\sum_{j=1}^nt_j\Phi_j(f(A_j))\,.
		\end{align*}
		\item[(ii)] Let $h(t)=1$ in Theorem \ref{t5.5}, then we have
		\begin{align*}
		f\left(mI_{\mathcal K}+MI_{\mathcal K}-\sum_{j=1}^nt_j\Phi_j(A_j)\right)\le 2\big(f(m)+f(M)\big)I_{\mathcal K}-\sum_{j=1}^nt_j\Phi_j(f(A_j))\,.
		\end{align*}
	\end{enumerate}
\end{corollary}
With similar proof of Theorem \ref{t5.5}, we have the following proposition.
\begin{proposition}\label{p5.7}
	Let $A_1, A_2, \cdots, A_n\in\mathbb{B}(\mathcal H)$ be self-adjoint operators with spectra in $[m, M]$ for some scalars and $\Phi_1, \Phi_2, \cdots, \Phi_n\in \mathbf{P}[\mathbb{B}(\mathcal{H}), \mathbb{B}(\mathcal{K})]$ positive linear maps with $\sum_{j=1}^n\Phi_j(I_{\mathcal H})=I_{\mathcal K}$. If $f$ on $[m, M]$ is operator $h$-convex function and $h$ on the interval $J$ is operator convex function, then
	\begin{align*}
	f & \left(mI_{\mathcal K}+MI_{\mathcal K}-\sum_{j=1}^n\Phi_j(A_j)\right)\nonumber\\
	& \le h\left(\frac{\sum_{j=1}^n\Phi_j(A_j)-mI_{\mathcal K}}{M-m}\right)f(m)+h\left(\frac{MI_{\mathcal K}-\sum_{j=1}^n\Phi_j(A_j)}{M-m}\right)f(M)\\
	& \le (\mu_h+2\nu_h)\big(f(m)+f(M)\big)I_{\mathcal K}-\sum_{j=1}^n\Phi_j(f(A_j))\,.\nonumber
	\end{align*}
\end{proposition}
\begin{corollary}\cite[Theorem 1]{MPP}
	Let $h(t)=t$ in Proposition \ref{p5.7}, then we have
	\begin{align*}
	f & \left(mI_{\mathcal K}+MI_{\mathcal K}-\sum_{j=1}^n\Phi_j(A_j)\right)\\
	& \le \left(\frac{\sum_{j=1}^n\Phi_j(A_j)-mI_{\mathcal K}}{M-m}\right)f(m)+\left(\frac{MI_{\mathcal K}-\sum_{j=1}^n\Phi_j(A_j)}{M-m}\right)f(M)\\
	& \le f(m)I_{\mathcal K}+f(M)I_{\mathcal K}-\sum_{j=1}^n\Phi_j(f(A_j))\,.
	\end{align*}
\end{corollary}
\begin{theorem}\label{t55}
	Suppose that $A_j\in\mathbb{B}_h(\mathcal H)$ with $\sigma(A_j)\in[m, M]~(m<M)$, $\Phi_j\in\mathbf{P}_N[\mathbb{B}(\mathcal H), \mathbb{B}(\mathcal K)]~(j=1, \cdots, n)$ and $t_1, \cdots, t_n\ge 0$ such that $\sum_{j=1}^nt_j=1$. If $f, g\in \mathcal{C}([m, M])$ and $F(u, v)$ is a real valued continuous function defined on $U\times V$, where $f[m, M]\subset U, g[m, M]\subset V$ and $F$ is an operator monotone function relative to the first component $u$ and $f$ is a non-negative operator monotone $h$-convex and $h$ is super-multiplicative operator convex function on $J$, then
	\begin{align}\label{eq444}
	F&\left[\sum_{j=1}^nh(t_j)\Phi_j(f(A_j)), g\left(\sum_{j=1}^nt_j\Phi_j(A_j)\right)\right]\\
	&\le \max_{m\le t \le M}F\left[\mu_h\mu_ft+\mu_h\nu_f+\nu_h\big(f(m)+f(M)\big), g(t)\right]I_{\mathcal K}\,.\nonumber
	\end{align}
\end{theorem}
\begin{proof}
	With the above assumptions and similar proof of previous theorem, we have
	\begin{align*}
	\sum_{j=1}^{n}h(t_j)\Phi_j(f(A_j)) \le & \mu_h\left[\mu_f\sum_{j=1}^nt_j\Phi_j(A_j)+\nu_fI_{\mathcal K}\right]+\nu_h\big(f(m)+f(M)\big)I_{\mathcal K}\,.
	\end{align*}
	Consequently,
	{\footnotesize 
		\begin{align*}
		F&\left[\sum_{j=1}^nh(t_j)\Phi_j(f(A_j)), g\Big(\sum_{j=1}^nt_j\Phi_j(A_j)\Big)\right]\\
		& \le F\left[\mu_h\Big[\mu_f\sum_{j=1}^nt_j\Phi_j(A_j)+\nu_fI_{\mathcal K}\Big]+\nu_h\big(f(m)+f(M)\big)I_{\mathcal K}, g\Big(\sum_{j=1}^nt_j\Phi_j(A_j)\Big)\right]\\
		&\le \max_{m\le t \le M}F\left[\mu_h\mu_ft+\mu_h\nu_f+\nu_h\big(f(m)+f(M)\big), g(t)\right]I_{\mathcal K}\,,
		\end{align*}
	}
	and we have the desired inequality \eqref{eq444}.
\end{proof}
\begin{theorem}\label{t56}
	With assumptions of previous theorem, if $f$ is operator $h$-convex on $[m, M]$ and $h$ is operator convex function on $J$, then for every $\alpha\in\mathbb{R}$
	\begin{equation}\label{eq56}
	\sum_{j=1}^nh(t_j)\Phi_j(f(A_j)) \le \alpha g\left(\sum_{j=1}^nt_j\Phi_j(A_j)\right)+\beta I_{ K}\,,
	\end{equation}
	where 
	$$
	\beta=\max_{m\le t \le M}\left\{\mu_h\mu_ft+\mu_h\nu_f+\nu_h\big(f(m)+f(M)\big)-\alpha g(t)\right\}\,.
	$$
	In addition, 
	\begin{enumerate}
		\item[(i)] If $\alpha g$ is concave, then 
		$$
		\beta\ge\max_{s\in\{m, M\}}\left\{\mu_h f(s)+\nu_h\big(f(m)+f(M)\big)-\alpha g(s)\right\}\,.
		$$
		\item[(ii)] If $\alpha g$ is strictly convex differentiable, then 
		$$
		\beta \le \mu_hf(s)-\alpha g(s)+|\mu_h\mu_f-\alpha g'(s)|(M-m)+\nu_h\big(f(m)+f(M)\big)\,,
		$$
	\end{enumerate}
	where $s\in\{m, M\}$.
\end{theorem}
\begin{proof}
	Let $\alpha\in\mathbb{R}$. Define $F(u, v)=u-\alpha v$. Since $F$ is operator monotone on a first variable $u$, hence by Theorem \ref{t55}, we have 
	\begin{align*}
	\sum_{j=1}^n&h(t_j)\Phi_j(f(A_j))-\alpha g\left(\sum_{j=1}^nt_j\Phi_j(A_j)\right)\\
	&\le \max_{m\le t \le M}\left\{\mu_h\mu_ft+\mu_h\nu_f+\nu_h\big(f(m)+f(M)\big)-\alpha g(t)\right\}I_{\mathcal K}\,,
	\end{align*}
	and we have the desired inequality \eqref{eq56}.
	
	Put $\Psi(t)=\mu_h(\mu_ft+\nu_f)+\nu_h\big(f(m)+f(M)\big)-\alpha g(t)$. In the case (i), if $\alpha g(t)$ is concave on $[m, M]$, then $\Psi$ is convex on $[m, M]$ and therefore 
	\begin{align*}
	\beta=&\max_{m\le t \le M}\Psi(t)=\max\{\Psi(m), \Psi(M)\}\\
	=& \max_{s\in\{m, M\}}\left\{\mu_h f(s)+\nu_h\big(f(m)+f(M)\big)-\alpha g(s)\right\}\,.
	\end{align*}
	In the case (ii), if $\alpha g$ is strictly convex differentiable, then there exists $t_0\in[m, M]$ such that $\beta=\max_{m\le t \le M}\Psi(t)=\Psi(t_0)$. So, if $s\in\{m, M\}$, then 
	\begin{align*}
	\beta=& \mu_h(\mu_ft_0+\nu_f)+\nu_h\big(f(m)+f(M)\big)-\alpha g(t_0)\\
	=& \mu_h f(s)+\mu_h\mu_f(t_0-s)-\alpha g(t_0)+\nu_h\big(f(m)+f(M)\big)\\
	=& \mu_h f(s)-\alpha g(s)+\mu_h\mu_f(t_0-s)-\alpha g(t_0)+\alpha g(s)+\nu_h\big(f(m)+f(M)\big)\\
	\le & \mu_hf(s)-\alpha g(s)+\mu_h\mu_f(t_0-s)-\alpha g'(s)(t_0-s)+\nu_h\big(f(m)+f(M)\big)\\
	&\quad\hspace*{5cm}\hfill(\alpha g\text{~is~strictly~convex~differentiable})\\
	\le & \mu_hf(s)-\alpha g(s)+|\mu_h\mu_f-\alpha g'(s)|(M-m)+\nu_h\big(f(m)+f(M)\big)\,.
	\end{align*}
\end{proof}
\begin{corollary}[Complementary inequality of Jensen's inequality]
	Let $A_j$, $\Phi_j$ and $t_j~( j=1, \cdots, n)$ be as in Theorem \ref{t56}. If $\in \mathcal{C}([m,M])$ is a function which is nonnegative, strictly $h$-convex and twice differentiable, then for every $\alpha \in \mathbb{R}^+$
	\begin{equation}\label{eq58}
	\sum_{j=1}^nh(t_j)\Phi_j(f(A_j)) \le \alpha f\left(\sum_{j=1}^nt_j\Phi_j(A_j)\right)+\beta I_{\mathcal K}\,,
	\end{equation}
	where $\beta=\mu_h\mu_ft_0+\mu_h\nu_f+\nu_h\big(f(m)+f(M)\big)-\alpha f(t_0)$ and 
	$$
	t_0=\begin{cases}
	f'^{-1}\left(\frac{1}{\alpha}\mu_h\mu_f\right) &; ~\text{if}~\alpha f'(m)<\mu_h\mu_f<\alpha f'(M)\,,\\
	m &; ~\text{if}~\alpha f'(m)\ge \mu_h\mu_f\,,\\
	M &; ~\text{if}~\alpha f'(M)\le \mu_h\mu_f\,.
	\end{cases}
	$$
\end{corollary}


\begin{thebibliography}{}
	%
	%
	\bibitem{Bre} W.W. Breckner, Stetigkeitsaussagen f\"ur eine Klasse verallgemeinerter konvexer funktionen in topologischen linearen R\"aumen, Publ. Inst. Math. \textbf{23} (1978) 13--20.
	
	\bibitem{DDNT} V. Darvish, S.S. Dragomir, H.M. Nazari and A. Taghavi, Some inequalities associated with the Hermite-Hadamard inequalities for operators $h$-convex functions, Acta et Commentationes Universitatis Tratuensis de Mathematica, \textbf{21} (2017), 287--297.
	
	\bibitem{DK} T.H. Dinh and K.T.B. Vo, Some inequalities for operator $(p, h)$-convex functions, Linear Multilinear A. \textbf{66} (2018), 580--592.
	
	\bibitem{DPP} S.S. Dragomir, J. Pe\v{c}ari\'{c} and L.E. Persson, Some inequalities of Hadamard type, Soochow J. Math. \textbf{21} (1995) 335--341.
	
	\bibitem{FMPS} T. Furuta, J. Mi\'{c}i\'{c} Hot, J.E. Pe\v{c}ari\'{c} and Y. Seo, Mond-Pe\v{c}ari\'{c} method in operator inequalities, Element, Zagreb, 2005.
	
	\bibitem{GL} E.K. Godunova and  V.I. Levin, Neravenstva dlja funkcii \v{s}irokogo klassa, soder\v{z}a\v{s}\v{c}ego vypuklye, monotonnye i nekotorye drugie vidy funkcii, in: Vy\v{c}islitel. Mat. i. Mat. Fiz. Me\v{z}vuzov. Sb. Nau\v{c}. Trudov, MGPI, Moskva, (1985)
	138--142.
	
	\bibitem{MPP} A. Matkovi\'{c}, J. Pe\v{c}ari\'{c} and I. Peri\'{c}, A variant of Jensen's inequality of Mercer's type for operators with applications, Linear Algebra Appl. \textbf{418} (2006), 551--564.
	
	\bibitem{AMDM} A. McD. Mercer, A variant of Jensen's inequality, J. Ineq. Pure and Appl. Math., \textbf{4} (2003), Article 73. [ONLINE: \texttt{http://www.emis.de/journals/JIPAM/article314.html}].
	
	\bibitem{RHM} J. Rooin, S. Habibzadeh and M.S. Moslehian, Jensen inequalities for $P$-class functions, Period Math. Hung. \textbf{77} (2018), 261--273.
	\bibitem{VAR} S. Varo\v{s}anec, On $h$-convexity, J. Math. Anal. Appl., \textbf{326} (2007), 303--311.
	
\end{thebibliography}
\end{document}